\documentclass[letter]{amsart}
\usepackage[foot]{amsaddr}

\title[The Euler-Lagrange Equations as the Gradient]{Interpreting the Euler-Lagrange Equations as the Gradient of the Action Functional}

\author[Montek Singh Gill]{Montek Singh Gill}
\email{montekg@umich.edu}

\usepackage[letterpaper,hmargin=1.25in,vmargin=1.2in]{geometry}
\usepackage{amsmath}
\usepackage{amssymb}
\usepackage{amsthm}

\theoremstyle{theorem}
\newtheorem{Theorem}{Theorem}[section]

\newtheorem{Proposition}[Theorem]{Proposition}

\theoremstyle{definition}
\newtheorem{Definition}[Theorem]{Definition}
\newtheorem{Remark}[Theorem]{Remark}
\newtheorem{Example}[Theorem]{Example}

\newcommand{\N}{\mathbb{N}}
\newcommand{\R}{\mathbb{R}}

\newcommand{\E}{\mathbb{E}}

\newcommand{\Sph}{\mathbb{S}}

\usepackage{tikz}
\usepackage{tikz-cd}
\usepackage{pgfplots}
\usetikzlibrary{decorations.markings}
\usetikzlibrary{arrows,shapes,positioning}
\usetikzlibrary{matrix}
\usetikzlibrary{cd}
\usetikzlibrary{arrows.meta}

\usepackage{listings}
\usepackage{color}

\usepackage{bbm}
\usepackage{mathrsfs}

\usepackage[]{hyperref}

\begin{document}

\subjclass[2010]{57P99, 58A05, 49N99.}

\keywords{diffeological space, tangent space, path space, calculus of variations, Euler-Lagrange equations}

\maketitle

\begin{abstract}
We study the smooth path spaces of Euclidean spaces $\R^N$, as diffeological spaces. We show that the tangent spaces of the free path space $\mathscr{P}$ are isomorphic to $\mathscr{P}$ itself, and that the tangent spaces of the space $\mathscr{P}_{\mathbf{p}, \mathbf{q}}$ of paths with fixed endpoints $\mathbf{p}$ and $\mathbf{q}$ are isomorphic to the smooth loop space of $\R^N$ based at the origin. We also define cotangents and gradients of smooth maps from these path spaces, and then show that, in the case of the action functional which arises in the calculus of variations, the gradient is precisely the path formed out of the terms of the Euler-Lagrange equations. We show that solutions of the Euler-Lagrange equations correspond precisely to the zeros of the gradient, and also provide analogous interpretations for the constrained Euler-Lagrange equations. This gives an illuminating geometric perspective on these equations. Finally, we illustrate the theory with several concrete examples from geometry, mechanics and machine learning.
\end{abstract}

\tableofcontents

\section{Introduction}

In the calculus of variations, and in its applications in mechanics and elsewhere, the Euler-Lagrange equations are a system of second-order ordinary differential equations. In this work, we give an illuminating interpretation of these equations. These equations arise when we look for the stationary points, defined appropriately, of the action functional, a functional which acts on paths in Euclidean space with fixed endpoints. To be more specific, let us fix $a, b \in \R$ with $a < b$, and also points $\mathbf{p}, \mathbf{q} \in \R^N$, for some $N \ge 1$. Let also $L \colon \R^{2N+1} \to \R$ be a smooth map, referred to as the Lagrangian. Then, given any path $\gamma \colon [a,b] \to \R^N$ with $\gamma(a) = \mathbf{p}$ and $\gamma(b) = \mathbf{q}$, the associated action, denoted $S$, is defined as
\[
S := \int_a^b L(\gamma(t), \gamma'(t), t)\,dt
\]
If we let $\mathscr{P}_{\mathbf{p}, \mathbf{q}}$ denote the set of such paths, this gives us a map:
\[
S \colon \mathscr{P}_{\mathbf{p}, \mathbf{q}} \to \R
\]
The Euler-Lagrange equations determine paths $\gamma$ at which $S$ is stationary. There is no obvious smooth structure however on $\mathscr{P}$, but we can still consider perturbations of the path $\gamma$, by adding $h\eta$ for some other path $\eta$ which is zero at the endpoints, and then determine necessary conditions for the stationarity of $S$ under these perturbations. These necessary conditions are precisely the Euler-Lagrange equations
\begin{align*}
\frac{\partial \mathcal{L}}{\partial x_1} &= \frac{d}{dt}\frac{\partial \mathcal{L}}{\partial \dot{x}_1} \\
&\quad\vdots \\
\frac{\partial \mathcal{L}}{\partial x_N} &= \frac{d}{dt}\frac{\partial \mathcal{L}}{\partial \dot{x}_N}
\end{align*}
where, as usual, $x_1, \dots, x_n, \dot{x}_1, \dots, \dot{x}_n, t$, in that order, represent the arguments of $L$. While there is no smooth structure, in the sense of smooth manifolds, on the path spaces of $\R^N$ (as these spaces must clearly be infinite dimensional), they do possess smooth structures in the sense of diffeological spaces. In this work, we endow them with these structures, and then compute their tangent spaces. We find that the tangent spaces of $\mathscr{P}_{\mathbf{p}, \mathbf{q}}$ are isomorphic to the smooth loop space of $\R^N$ based at the origin, while those of $\mathscr{P}$, the free path space of $\R^N$, are isomorphic to $\mathscr{P}$ itself. Next, we define gradients of smooth functions on the path spaces, and find that, in an appropriate sense which generalizes the case of ordinary finite dimensional manifolds, the gradient of the action functional $S$ above, is, as an element of $\mathscr{P}$, precisely the path:
\[
t \mapsto \left(\frac{\partial \mathcal{L}}{\partial x_1} - \frac{d}{dt}\frac{\partial \mathcal{L}}{\partial \dot{x}_1}, \cdots, \frac{\partial \mathcal{L}}{\partial x_N} - \frac{d}{dt}\frac{\partial \mathcal{L}}{\partial \dot{x}_N} \right)
\]
That is, the terms which arise in the Euler-Lagrange equations collectively assemble together to form a tangent to the path space, and this tangent is precisely the gradient of the action functional. Moreover, we show that solutions of the Euler-Lagrange equations correspond precisely to the zeros of the gradient. We also show that analogous interpretations can be made for the constrained Euler-Lagrange equations. \\

In the final section, we give several concrete examples in various scenarios, which provide clarity and intuition for the above theory. They also demonstrate how one can use the direction of the Euler-Lagrange tangent paths to determine whether one has a local minimum or a local maximum.

\section{Path Spaces as Diffeological Spaces}

We begin with diffeological spaces in general, and then discuss the specific spaces of interest to us, namely the path spaces of Euclidean spaces.

\begin{Definition}
A \emph{diffeological space} is a set $X$ together with a specified set $\mathcal{D}_X({\widehat U})$, a subset of the set of functions $\widehat U \to X$, referred to as \emph{plots}, for each open set $\widehat U$ of $\R^n$, for any $n \in \N$, such that:
\begin{itemize}
	\item[(i)] [Covering] For any $\widehat U \subseteq \R^n$, every constant function $\widehat U \to X$ is a plot.
	\item[(ii)] [Smooth compatibility] For all $\widehat U \subseteq \R^n$ and $\widehat V \subseteq \R^m$, if $\widehat U \to X$ is a plot and $\widehat V \to \widehat U$ is smooth, then the composite $\widehat V \to \widehat U \to X$ is also a plot.
	\item[(iii)] [Sheaf condition] For all $\widehat U$ and open covers $\{\widehat U_i\}$ covers of $U$, if $\widehat U \to X$ is a function such that each restriction $\widehat U_i \to X$ is a plot, then $\widehat U \to X$ is a plot.
\end{itemize}
A function $X \to Y$ between diffeological spaces is \emph{smooth} if, for every plot $\widehat U \to X$ of $X$, the composite $\widehat U \to X \to Y$ is a plot of $Y$.
\end{Definition}

Diffeological spaces and smooth maps between them clearly form a category, which is denoted by \textbf{Diff}. We shall denote the hom-sets in $\textbf{Diff}$ by $\mathrm{C}^\infty(X,Y)$. Diffeological spaces generalize the notion of smooth manifolds, as the next result demonstrates.

\begin{Proposition}
\label{prop:smooth_manifolds_as_diffeological_spaces}
Any smooth manifold, with or without boundary, $M$ is a diffeological space with $\mathcal{D}_M(\widehat{U})$ set to be the smooth, in the sense of smooth manifolds, maps $\widehat U \to M$. This produces a functor
\[
\textbf{\emph{SmoothMan}} \to \textbf{\emph{Diff}}
\]
from the category of smooth manifolds to the category of diffeological spaces. Moreover, a function $M \to N$ between smooth manifolds, with our without boundary, is smooth in the sense of diffeological spaces if and only if it is in the sense of smooth manifolds; that is, the functor above is full and faithful.
\end{Proposition}

\begin{proof}
This is standard. See, e.g., \cite{DiffeologyBook}.
\end{proof}

Note that if $X$ is a diffeological space and $p \colon \widehat U \to X$ is a plot, then $p$ is itself smooth in the sense of diffeological spaces, where $\widehat U$ has the diffeological structure described in the result above. \\

A diffeological space, being a space, ought to carry a natural topology, and it does. This is the topology that consists of the D-open subsets, where a subset $U$ of $X$ is \textit{D-open} if, for every plot $p \colon \widehat U \to X$, $p^{-1}(U)$ is open. We call this the \textit{D-topology} and denote it by $\mathcal{T}(\mathcal{D}_X)$. \\

The next result shows that, categorically, diffeological spaces behave much better than smooth manifolds, and even better than topological spaces.

\begin{Proposition}
The category \textbf{\emph{Diff}} is complete, cocomplete and cartesian closed.
\end{Proposition}

\begin{proof}
This is also standard. See, e.g., \cite{DiffeologyBook}.
\end{proof}

Here are some example constructions:
\begin{itemize}
	\item Given a diffeological space $X$ and a subset $A$, $A$ becomes a \textit{diffeological subspace} with the following subspace diffeology: a function $\widehat U \to A$ is a plot of $A$ if and only if the composite $\widehat U \to A \to X$ is a plot of $X$; that is, the plots are exactly those of $X$ which have image in $A$. Given a diffeological space $X$ and a diffeological subspace $A$, there are two topologies that $A$ can inherit: the D-topology $\mathcal{T}(\mathcal{D}_A)$ from its diffeological structure, and the subspace topology from the D-topology $\mathcal{T}(\mathcal{D}_X)$ of $X$. In general, these two do not coincide. If they do, we say that $A$ is an \textit{embedded} diffeological subspace.
	\item Given diffeological spaces $X$ and $Y$, the \textit{diffeological product space} $X \times Y$ is the set-level product together with the following diffeology: a function $\widehat U \to X \times Y$ is a plot if and only if the component functions $\widehat U \to X$ and $\widehat U \to Y$ are plots of $X$ and $Y$, respectively.
	\item Given diffeological spaces $X$ and $Y$, the \textit{diffeological internal hom} of $X$ and $Y$ is $\mathrm{C}^\infty(X,Y)$ with the following diffeology: a function $\widehat U \to \mathrm{C}^\infty(X,Y)$ is a plot if and only if the adjoint function $\widehat U \times X \to Y$ is smooth (with respect to the product diffeology on $\widehat U \times X$). We shall denote the internal homs by $[X,Y]$. Note that, given diffeological spaces $X$, $Y$ and $Z$, and smooth maps $f \colon X \to Y$, $g \colon Y \to Z$, the precomposition and postcomposition maps $- \circ f \colon [Y,Z] \to [X,Z]$ and $g \circ - \colon [X,Y] \to [X,Z]$ are smooth maps. This follows from a purely formal argument using cartesian closure.
\end{itemize}

Now we can introduce the specific diffeological spaces of interest to us. Fix some natural number $N \ge 1$, real numbers $a, b$ with $a < b$ and points $\mathbf{p}, \mathbf{q} \in \R^N$. Consider the set of smooth paths $[a,b] \to \R^N$.

\begin{Definition}
We let $\mathscr{P}$ denote the internal hom $[[a.b], \R^N]$, the diffeological space of smooth paths $\gamma \colon [a,b] \to \R^N$. We also let $\mathscr{P}_{\mathbf{p}, \mathbf{q}}$ denote the diffeological subspace of those paths where $\gamma(a) = \mathbf{p}$ and $\gamma(b) = \mathbf{q}$.
\end{Definition}

By definition of the internal hom and subspace diffeologies, a function $\widehat U \to \mathscr{P}$, or $\widehat U \to \mathscr{P}_{\mathbf{p}, \mathbf{q}}$, is a plot exactly when the adjoint function $\widehat U \times [a,b] \to \R^N$ is smooth. The next couple examples illustrate that some facts which one would expect to be true of a smooth structure on $\mathscr{P}$ are indeed true with the diffeological structure defined above. 

\begin{Example}
\label{ex:D_topology_of_path_space}
Let $\gamma \in \mathscr{P}$ and let $U$ be a D-open subset of $\mathscr{P}$. If $\eta \in \mathscr{P}$ is another path, one would expect that, for sufficiently small $\varepsilon$, $\gamma + \varepsilon\eta$ is also in $U$. This is indeed true. To see this, set $\widehat U = (-\varepsilon, \varepsilon)$ and consider the map $(-\varepsilon, \varepsilon) \times [a,b] \to \R^N : (h,t) \mapsto \gamma(t) + h\eta(t)$. This is a smooth map, and so, by the definition of the diffeology of $\mathscr{P}$, we have an associated plot $\widehat U \to \mathscr{P} \colon h \mapsto \gamma + h\eta$. As $U$ is D-open, the pre-image of $U$ under this plot is open in $\widehat U$. The preimage contains $0$ as $0$ maps to $\gamma$, and so a neighbourhood of $0$ is in the preimage, which gives us the desired result.
\end{Example}

\begin{Example}
\label{ex:smooth_function_on_path_space}
Fix some path $\gamma \in \mathscr{P}$ and consider the following function:
\[
f \colon \mathscr{P} \to \R : \alpha \mapsto \int_a^b \gamma(t) \cdot \alpha(t)\,dt
\]
One would expect this to be a smooth function on the path space, and this is indeed true. To see this, let $p \colon \widehat{U} \to \mathscr{P}$ be a plot; we need to show that the composite $f \circ p : \widehat{U} \to \mathscr{P} \to \R$ is smooth. By definition of the diffeology of $\mathscr{P}$, the adjoint $\widetilde p \colon \widehat{U} \times [a,b] \to \R^N$ is smooth. Moreover, the composite $f \circ p$ is given by
\[
\widehat{U} \to \R : \mathbf{u} \mapsto \int_a^b \gamma(t) \cdot \widetilde p(\mathbf{u}, t)\,dt
\]
and, knowing that $\widetilde p$ is smooth, this is indeed smooth. By a similar argument, if $L \colon \R^{2N+1} \to \R$ is smooth, then the following is a smooth function on paths:
\[
\mathscr{P} \to \R : \alpha \mapsto \int_a^b L(\alpha(t), \alpha'(t), t)\,dt
\]
\end{Example}

One more note before moving on. It is clear that the path spaces $\mathscr{P}$ and $\mathscr{P}_{\mathbf{p}, \mathbf{q}}$ are infinite dimensional. In fact, we can think of the inclusion
\[
\mathscr{P}_{\mathbf{p}, \mathbf{q}} \hookrightarrow \mathscr{P}
\]
as an infinte dimensional variant of the inclusion
\[
H \hookrightarrow \R^{n+k}
\]
where $n, k > 0$ and $H$ is the affine subspace $\{(x_1,\dots,x_{n+k}) \mid x_{n+1} = c_{n+1}, \dots, x_{n+k} = c_{n+k}\}$ for some constants $c_{n+1}, \dots, c_{n+k}$. We will use this later.

\section{Tangent Spaces of Path Spaces}

Our goal now is to define tangent spaces to our path spaces. Once more, we must first discuss general diffeological theory, and then consider our specific spaces. Note that it is not obvious how one should define tangent spaces for general diffeological spaces. In \cite{ChristensenWu}, Christensen and Wu compare and discuss several approaches. We shall adopt a definition which is intermediate between what are called ``internal tangent spaces'' and ``external tangent spaces'' in \cite{ChristensenWu}; in fact, Christensen and Wu, in a section on alternative approaches, mention another approach which we shall show later to be exactly equivalent to the approach that we take below. Our approach is also very close to the approach used by Vincent in \cite{Vincent}. Moreover, our approach satisfies the basic requirements that any reasonable notion of tangent space should satisfy (see, e.g., \cite{Vincent} for a list of such requirements).

\subsection{Tangent Spaces of Diffeological Spaces}

First, we must define spaces of germs of smooth functions.

\begin{Definition}
Let $X$ be a diffeological space and $x \in X$. The \emph{germs of smooth functions of $X$ at $x$}, denoted $\mathrm{G}_xX$, is defined as
\[
\mathrm{G}_xX := \underset{U \, \ni \, x}{\mathrm{colim}} \, [U, \R]
\]
where the colimit is taken in diffeological spaces, $U$ runs over all D-open subsets, equipped with sub-diffeologies, of $X$ which contain $x$, and the maps in the colimit are restrictions along inclusions.
\end{Definition}

The space of germs $\mathrm{G}_xX$ is in fact a diffeological $\R$-algebra and can be described concretely as follows:
\begin{itemize}
	\item The underlying set consists of pairs $(U, f)$ where $U$ is a D-open neighbourhood of $x$ and $f$ is a smooth map $U \to \R$, taken under the equivalence relation where $(U, f) \sim (V, g)$ if and only if $f = g$ on some D-open neighbourhood of $x$ inside $U \cap V$. We typically denote the equivalence class $[(U,f)]$ by just $[f]$.
	\item The diffeology is as follows: a function $p \colon \widehat U \to \mathrm{G}_xX$ is a plot if and only if, for each $\mathbf{u} \in \widehat U$ there is some $\widehat V \subseteq \widehat U$ containing $\mathbf{u}$ and a D-open $U$ containing $x$ such that the restriction $p|_{\widehat V} \colon \widehat V \to \mathrm{G}_xX$ filters as $\widehat V \to [U, \R] \to G_xX$ for a plot $\widehat V \to [U,\R]$ of $[U,\R]$.
	\item Addition, multiplication and scalar multiplication are defined as:
\[
c[(U,f)] = (U, cf)
\]
\[
[(U,f)] + [(V,g)] = (U \cap V, f+g)
\]
\[
[(U,f)] \cdot [(V,g)] = (U \cap V, fg)
\]
These are clearly well-defined and an easy verification shows that they are also smooth. Moreover, the evaluation map $\mathrm{G}_xX \to \R$ sending $[f]$ to $f(x)$ is a well-defined smooth $\R$-algebra map. 
\end{itemize}

Now we can define tangent vectors, and tangent spaces, of diffeological spaces.

\begin{Definition}
\label{def:tangent_space_to_diffeological_space}
Let $X$ be a diffeological space and let $x \in X$. A \emph{tangent vector on $X$ at $x$} is a smooth derivation on $\mathrm{G}_xX$ which is representable by a path. That is, it is a map $v \colon \mathrm{G}_xX \to \R$ such that:
\begin{itemize}
	\item[(i)] [Linearity] $v$ satisfies $v(\lambda[f] + \mu[g]) = \lambda v([f] + \mu v([g])$.
	\item[(ii)] [Leibniz rule] $v$ satisfies $v([f][g]) = v([f])g(x) + f(x)v([g])$.
	\item[(iii)] [Smoothness] $v$ is smooth in the sense of diffeological spaces.
	\item[(iv)] [Representability] $v$ is a linear combination $\lambda_1 v_1 + \cdots \lambda_k v_k$ of derivations $v_i$ where, for each $i$, there exists a plot $p_i \colon (-\varepsilon_i, \varepsilon_i) \to X$, for some $\varepsilon_i > 0$, of $X$ such that $p_i(0) = x$ and $v_i([f]) = \frac{d}{dh}\Big\vert_{h=0}f(p_i(h))$.
\end{itemize}
The \emph{tangent space} $\mathrm{T}_x X$ is the set of all tangent vectors of $X$ at $x$. Clearly $\mathrm{T}_x X$ is a vector space under pointwise addition and pointwise scalar multiplication.
\end{Definition}

Given a smooth map $F \colon X \to Y$ between two diffeological spaces, at any point $x \in X$, we have an induced linear map $F_*$ between the two tangent spaces
\[
F_* \colon \mathrm{T}_x X \to \mathrm{T}_{F(x)} Y
\]
where $F_*(v)([f]) = F([f \circ F])$ for $v \in \mathrm{T}_x X$ and $[f] \in \mathrm{G}_{F(x)}Y$ (here, by $f \circ F$, we of course really mean $f \circ F|_{F^{-1}(\mathrm{dom}(f))}$). It is clear that $F_*(v)$ satisfies linearity, the Leibniz rule and representability. To see smoothness, we can express $F_*$ in another way. By precomposition with $F$, we get an induced smooth map $\mathrm{G}_{F(x)}Y \to \mathrm{G}_xX$, and the map $F_*$ is simply precomposition with this map. This makes it clear that $F_*(v)$ is indeed smooth if $v$ is. \\

Our first result on tangent spaces to diffeological spaces demonstrates that they are local objects.

\begin{Proposition}
\label{prop:tangent_spaces_are_local}
Let $X$ be a diffeological space and $A$ an embedded D-open diffeological subspace. For any $a \in A$, the inclusion map $i \colon A \to X$ induces an isomorphism:
\[
i_* \colon \mathrm{T}_aA \overset{\cong}\longrightarrow \mathrm{T}_aX
\]
\end{Proposition}

\begin{proof}
Because $A$ is embedded and D-open, a subset $U$ of $A$ is D-open in $A$ if and only if it is D-open in $X$. From this, it follows easily that the induced map $\mathrm{G}_aX \to \mathrm{G}_aA$ is an isomorphism of diffeological $\R$-algebras. The desired result then clearly follows from this.
\end{proof}

The following result shows that the tangent spaces to Euclidean spaces are as expected.

\begin{Proposition}
\label{prop:tangent_spaces_for_smooth_manifolds}
We have the following:
\begin{itemize}
	\item[\emph{(i)}] For every open $\widehat U$ in $\R^n$, and $\mathbf{p} \in \widehat U$, we have a natural isomorphism
\[
\mathrm{D} \colon \R^n \overset{\cong}\longrightarrow \mathrm{T}_{\mathbf{p}}\widehat U
\]
where $\mathrm{D}\mathbf{v}$ is the derivation $[f] \mapsto \frac{d}{dh}f(\mathbf{p} + h\mathbf{v})$.
	\item[\emph{(ii)}] For any smooth $n$-manifold $M$ and $p \in M$, we have that:
\[
\mathrm{T}_pM \cong \R^n
\]
In fact, as a real vector space, $\mathrm{T}_pM$ is exactly equal to the tangent space in the sense of smooth manifolds.
\end{itemize}
\end{Proposition}

\begin{proof}
(i): An easy check from the definitions shows that the D-topology on $\widehat U$ is equivalent to the standard Euclidean topology, and so the colimit in constructing the space of germs $\mathrm{G}_{\mathbf{p}}\widehat U$ is in fact one over all open $U$ in $\widehat U$ containing $\mathbf{p}$. Moreover, by Proposition \ref{prop:smooth_manifolds_as_diffeological_spaces}, given an open $\widehat V$ in $\widehat U$, a map $\widehat V \to \R$ is smooth in the sense of diffeological spaces if and only if it is smooth in the Euclidean sense. Now, by the analogous result from smooth manifold theory for the tangent spaces to $\widehat U$, we know that each $\mathrm{D}\mathbf{v}$ is in fact a derivation on $\mathrm{G}_{\mathbf{p}}\widehat U$. It also clearly representable, using the path $h \mapsto \mathbf{p} \to h\mathbf{v}$. We claim that each $\mathrm{D}\mathbf{v}$ is also automatically smooth; at which point the desired result follows from the analogous result from smooth manifold theory. To see that $\mathrm{D}\mathbf{v}$ is smooth, first, note that any function $v \colon \mathrm{G}_{\mathbf{p}}\widehat U \to \R$ is smooth if and only if each restriction $[\widehat V, \R] \to \R$, for open neighbourhoods $\widehat V$ of $\mathbf{p}$ contained in $\widehat U$, is smooth. Because directional derivatives may be written as a linear combinations of partial derivatives, without loss of generality, we may assume that $\mathbf{v}$ is the $i^{\text{th}}$ standard basis vector of $\R^n$, for $i = 1, \dots, n$. In this case, the restriction $[\widehat V, \R] \to \R$ of $\mathrm{D}\mathbf{v}$ is the partial derivative operator $f \mapsto (\partial_if)(\mathbf{p})$. Let $p \colon \widehat W \to [\widehat V, \R]$ be a plot, for some open $\widehat W \subseteq \R^m$. We need to show that the composite $(\partial_i-)(\mathbf{p}) \circ p \colon \widehat W \to [\widehat V, \R] \to \R$ is smooth. As $p$ is a plot, we know that the adjoint $\widetilde{p} \colon \widehat W \times \widehat V \to \R$ is smooth. Then $\partial_{m+i}\widetilde{p} \colon \widehat W \times \widehat V \to \R$ is also smooth. Finally, it is clear that $(\partial_{m+i}\widetilde p)(\mathbf{w}, \mathbf{p}) = (\partial_ip(\mathbf{w}))(\mathbf{p})$, so that $(\partial_i-)(\mathbf{p}) \circ p$ is the composite $\widehat W \to \widehat W \times \widehat V \to \R$ where the first map is $\mathbf{w} \mapsto (\mathbf{w}, \mathbf{p})$ and the second map is $\partial_{m+i}\widetilde{p}$, and so $(\partial_i-)(\mathbf{p}) \circ p$ is smooth, as desired. \\

(ii): Let $\varphi \colon \widehat U \to U$ be a chart, where $U$ contains $p$; say $\varphi(\mathbf{p}) = p$. As $U$ is clearly an embedded D-open subspace, by Proposition \ref{prop:tangent_spaces_are_local}, we have that $\mathrm{T}_pM \cong \mathrm{T}_pU$ via the inclusion $i \colon U \to M$. Moreover, as $\varphi$ is an isomorphism of diffeological spaces, we have that $\mathrm{T}_pU \cong \mathrm{T}_{\mathbf{p}} \widehat U$. Finally, as in (i) above, $\mathrm{T}_{\mathbf{p}}\widehat U \cong \R^n$. \\

In fact, it is easy to see that, for any smooth manifold $M$, the D-topology coincides with the given topology, and moreover that, given any open $U$ in $M$, the functions $U \to \R$ are smooth in the sense of diffeological spaces if and only if they are smooth in the sense of smooth manifolds. It follows that the germs of smooth functions in the diffeological sense coincide with the germs of smooth functions in the manifold sense; it is only the diffeology that is added in the former case. It follows that, if we let $\mathrm{T}_p^{\text{man}}M$, for $p \in M$, denote the tangent space in the sense of smooth manifolds, then $\mathrm{T}_pM$ is a subspace of $\mathrm{T}_p^{\text{man}}M$ (the former contains the representable and smooth derivations, the latter all derivations, both on the same space of germs). Thus, we get the following diagram:
\begin{center}
\begin{tikzpicture}[node distance = 2cm]
\node [] (A) {$\mathrm{T}_{\mathbf{p}}\widehat U$};
\node [below of = A] (B) {$\mathrm{T}^{\text{man}}_{\mathbf{p}}\widehat U$};
\node [right of = A] (C) {$\mathrm{T}_pU$};
\node [below of = C] (D) {$\mathrm{T}^{\text{man}}_pU$};
\node [right of = C] (E) {$\mathrm{T}_pM$};
\node [below of = E] (F) {$\mathrm{T}^{\text{man}}_pM$};

\draw [->] (A) -- (B) node[midway, anchor=east]{$\subseteq$};
\draw [->] (A) -- (C) node[midway, anchor=south]{$\varphi_*$};
\draw [->] (B) -- (D) node[midway, anchor=north]{$\varphi_*$};
\draw [->] (C) -- (D) node[midway, anchor=east]{$\subseteq$};
\draw [->] (C) -- (E) node[midway, anchor=south]{$i_*$};
\draw [->] (E) -- (F) node[midway, anchor=west]{$\subseteq$};
\draw [->] (D) -- (F) node[midway, anchor=north]{$i_*$};
\end{tikzpicture}
\end{center}
All of the horizontal maps are isomorphisms. By (i) above, the leftmost vertical map is also an isomorphism. It follows that the other two vertical maps are also isomorphisms. Thus derivations on smooth manifolds are automatically representable and smooth.
\end{proof}

\subsection{Comparison With Other Approaches to Tangent Spaces}

In the literature, other approaches have been taken to the tangent spaces of diffeological spaces $X$. Here, we briefly mention a few. In \cite{Vincent}, Vincent uses exactly our notion of tangent space, except that: (i) $\R$ is used as the source of paths, as opposed to intervals $(-\varepsilon, \varepsilon)$ (ii) The derivations are taken on only global smooth functions $X \to \R$. Here (i) doesn't actually change anything, as one can always precompose with a smooth function $\R \to (-\varepsilon,\varepsilon)$ which has a unit derivative at the origin. On the other hand (ii) may lead to some differences. \\

In \cite{ChristensenWu}, given a diffeological space $X$ and $x \in X$, Christensen and Wu define the following two notions of tangent spaces:
\begin{itemize}
	\item The \textit{internal tangent space}, which we shall denote by $\mathrm{T}^{\text{int}}_xX$, is the colimit
\[
\underset{\small p \colon \widehat{U} \to X \, \text{centred at} \, x}{\mathrm{colim}} \mathrm{T}_0^{\text{man}}\widehat U
\]
where the colimit is taken over plots which are centred at $x$ (which means that $\widehat U$ is connected, $0 \in \widehat U$ and $p(0) = x$), and maps, represented by the obvious commuting triangles, between such plots. The tangent spaces of $\mathrm{T}_0^{\text{man}}\widehat U$ are the usual ones in the sense of smooth manifolds.
	\item The \textit{external tangent space}, which we shall denote by $\mathrm{T}^{\text{ext}}_xX$, is exactly our tangent space in Definition \ref{def:tangent_space_to_diffeological_space}, but without the representability requirement.
\end{itemize}
The authors also construct a natural map
\[
\mathrm{N} \colon \mathrm{T}^{\text{int}}_xX \to \mathrm{T}^{\text{ext}}_xX
\]
defined as follows: given a plot $p \colon \widehat U \to X$ centred at $x$ and a derivation $v \in \mathrm{T}_0^{\text{man}}\widehat U$, $\mathrm{N}([v])$ is the image of $v$ under $\mathrm{T}_0^{\text{man}}\widehat U \cong \mathrm{T}_0^{\text{ext}}\widehat U \to \mathrm{T}^{\text{ext}}_xX$. They observe that, in general, $\mathrm{N}$ is neither injective nor surjective. In a section on alternative approaches to tangent spaces, they suggest that one could consider the image of $\mathrm{N}$ as an alternate definition. This amounts to the internal tangent vectors, where two such vectors are identified if they yield the same directional derivative operators. The next result shows that our notion of tangent spaces in fact coincides exactly with this image.

\begin{Proposition}
For any diffeological space $X$ and $x \in X$, the tangent space $\mathrm{T}_xX$ is exactly the image $\mathrm{N}(\mathrm{T}^{\emph{int}}_xX) \subseteq \mathrm{T}^{\emph{ext}}_xX$.
\end{Proposition}

\begin{proof}
Let $v$ be a derivation in the image of $\mathrm{N}$, say $v = \mathrm{N}(w)$. As in \cite{ChristensenWu}, $w$ is necessarily a linear combination of vectors of the form $p_*^{\text{int}}([\frac{d}{dt}])$ where $p \colon \R \to X$ is a plot with $p(0) = x$ and $\frac{d}{dt}$ is the standard unit vector in $\mathrm{T}_0^{\text{int}}\R \cong \mathrm{T}_0^{\text{man}}\R$. Without loss of generality, we may assume that $w$ is exactly such a vector $p_*^{\text{int}}([\frac{d}{dt}])$. Consider the following diagram:
\begin{center}
\begin{tikzpicture}[node distance = 2cm]
\node [] (A) {$\R$};
\node [below of = A] (B) {$\R$};
\node [right of = A] (C) {$\mathrm{T}^{\text{int}}_0\R$};
\node [below of = C] (D) {$\mathrm{T}^{\text{ext}}_0\R$};
\node [right of = C] (E) {$\mathrm{T}^{\text{int}}_xX$};
\node [below of = E] (F) {$\mathrm{T}^{\text{ext}}_xX$};

\draw [->] (A) -- (B) node[midway, anchor=east]{$\text{id}$};
\draw [->] (A) -- (C) node[midway, anchor=south]{$1 \mapsto \frac{d}{dt}$};
\draw [->] (B) -- (D) node[midway, anchor=north]{$1 \mapsto \frac{d}{dt}$};
\draw [->] (C) -- (D) node[midway, anchor=east]{$\mathrm{N}$};
\draw [->] (C) -- (E) node[midway, anchor=south]{$p^{\text{int}}_*$};
\draw [->] (E) -- (F) node[midway, anchor=west]{$\mathrm{N}$};
\draw [->] (D) -- (F) node[midway, anchor=north]{$p^{\text{ext}}_*$};
\end{tikzpicture}
\end{center}
By assumption, $1 \in \R$ at the top-left corner, under the top two and rightmost maps, maps to $v$. Thus it also maps to $v$ under the leftmost and bottom two maps. It follows that $v$ is representable, and so in $\mathrm{T}_xX$. Moreover, practically the same argument, with the same diagram, but in reverse and with $\R$ replaced by $(-\varepsilon, \varepsilon)$, shows that if $v \in \text{T}_x^{\text{ext}}X$ is representable, then it is in the image of $\mathrm{N}$, as desired.
\end{proof}

\subsection{Tangent Spaces of Path Spaces}

Now let us consider our particular spaces of interest, namely $\mathscr{P}$ and $\mathscr{P}_{\mathbf{p}, \mathbf{q}}$, as defined in the previous section. Note that $\mathscr{P}$ is itself a real vector space, under pointwise addition and scalar multiplication. The fixed endpoint path space $\mathscr{P}_{\mathbf{p}, \mathbf{q}}$, on the other hand, is not a vector space under these operations; however, the space of paths $\mathscr{P}_{\mathbf{0}, \mathbf{0}}$ where the endpoints are fixed at the origin (that is, the smooth loop space based at the origin) is a vector space under pointwise addition and scalar multiplication. The next result computes the tangent spaces of the path spaces $\mathscr{P}$ and $\mathscr{P}_{\mathbf{p}, \mathbf{q}}$.

\begin{Theorem}
\label{thm:tangent_spaces_to_path_spaces}
For any $\gamma$ in $\mathscr{P}$, we have an isomorphism of vector spaces:
\[
\mathrm{D} \colon \mathscr{P} \overset{\cong}\longrightarrow \mathrm{T}_{\gamma}\mathscr{P}
\]
Similarly, for any $\gamma$ in $\mathscr{P}_{\mathbf{p}, \mathbf{q}}$, we have an isomorphism of vector spaces:
\[
\mathrm{D} \colon \mathscr{P}_{\mathbf{0},\mathbf{0}} \overset{\cong}\longrightarrow \mathrm{T}_{\gamma}\mathscr{P}_{\mathbf{p},\mathbf{q}}
\]
In either case, the isomorphism is defined by the formula $(\mathrm{D}\eta)([f]) = \frac{d}{dh}\Big\vert_{h = 0}f(\gamma + h\eta)$.
\end{Theorem}

Here, we are using the same notation $\mathrm{D}$ for both isomorphisms, and the same as we used for the isomorphism in Proposition \ref{prop:tangent_spaces_for_smooth_manifolds}(i). The reason is of course to emphasize that they all play the same role in their individual contexts; the given context will always make it clear which exact map we intend.

\begin{proof}
We shall prove the case of $\mathscr{P}$; the proof for $\mathscr{P}_{\mathbf{p}, \mathbf{q}}$ is entirely analogous. First, we need to check that $\mathrm{D}\eta$ is well-defined:
\begin{itemize}
	\item First, it follows by Example \ref{ex:D_topology_of_path_space} that, given a D-open $U$ containing $\gamma$, for sufficiently small $h$, $\gamma + h\eta$ is also in $U$.
	\item Next, we need to check that the function $h \mapsto f(\gamma + h\eta)$ is smooth. This is exactly the composite of the map $(-\varepsilon, \varepsilon) \to U : h \mapsto \gamma + h\eta$ (for an appropriate $\varepsilon > 0$), which is smooth as the adjoint $(-\varepsilon, \varepsilon) \times [a,b] \to \R^N$ is clearly smooth, and $f \colon U \to \R$, which is smooth by assumption, and so the function is itself smooth, as desired.
	\item Finally, we need to check that the derivative is independent of the choice of $f$. Suppose that $(U, f) \sim (V, g)$. Then $f = g$ on some D-open $W$ inside $U \cap V$. As per the first point above, for sufficiently small $h$, $\gamma + h\eta \in W$, and so $f(\gamma + h\eta) = g(\gamma + h\eta)$, which gives us the desired result.
\end{itemize}
Linearity and the Leibniz rule are clear. So is representability, via the path $h \mapsto \gamma + h\eta$ (which we showed just above to be smooth). We also need to check that $\mathrm{D}\eta$ is smooth. Consider a plot $\widehat U \to \mathrm{G}_\gamma\mathscr{P}$. We need to show that the composite $\widehat U \to G_\gamma\mathscr{P} \to \R$ is smooth. By the sheaf condition, it suffices to demonstrate this locally. By construction of the diffeology on $\mathrm{G}_\gamma\mathscr{P}$, for each $\mathbf{u} \in \widehat U$, there is some $\widehat V \subseteq \widehat U$ containing $\mathbf{u}$, a D-open $U$ containing $\gamma$ and a function $\widehat V \to [U, \R]$, a plot of $[U, \R]$, such that the adjoint $\widehat V \times U \to \R$ is smooth and the composite $\widehat V \to \widehat U \to \mathrm{G}_\gamma\mathscr{P}$ factors as $\widehat V \to [U, \R] \to G_\gamma\mathscr{P}$. Now, denoting the map $\widehat V \to [U,\R]$ by $\mathbf{v} \mapsto f_{\mathbf{v}}$, we need to show that the map
\[
F \colon \widehat V \to \R : \mathbf{v} \mapsto \frac{d}{dh}\Big\vert_{h=0}f_{\mathbf{v}}(\gamma + h\eta)
\] 
is smooth. We know that the map $\widehat V \times U \to \R : (\mathbf{v}, \alpha) \mapsto f_{\mathbf{v}}(\alpha)$ is smooth. As above, we may choose some $\varepsilon > 0$ such that $\gamma + h\eta \in U$ for each $h \in (-\varepsilon, \varepsilon)$; moreover, the map $(-\varepsilon, \varepsilon) \to \mathscr{P} : h \mapsto \gamma + h\eta$ is a plot of $\mathscr{P}$. By definition of the Euclidean and product diffeologies, we have that the map
\[
G \colon \widehat V \times (-\varepsilon, \varepsilon) \to \R : (\mathbf{v}, h) \mapsto f_{\mathbf{v}}(\gamma + h\eta)
\]
is smooth. In particular, then, the map
\[
\frac{\partial G}{\partial h} \colon \widehat V \times (-\varepsilon, \varepsilon) \to \R : (\mathbf{v}, h) \mapsto \frac{\partial}{\partial h}f_{\mathbf{v}}(\gamma + h\eta)
\]
is smooth. Finally, we see that $F(\mathbf{v})$ is exactly $\frac{\partial G}{\partial h}(\mathbf{v}, 0)$, and so $F$ is smooth, as desired. \\

We have now verified that, for each $\eta$, $\mathrm{D}\eta$ is a representable smooth derivation. It remains to check that the assignment $\eta \mapsto \mathrm{D}\eta$ is linear in $\eta$, and is bijective. It's easy to see, directly from the limit definition of derivatives, that $\mathrm{D}(\lambda \eta) = \lambda (\mathrm{D}\eta)$ for $\lambda \in \R$ and $\eta \in \mathscr{P}$. Additivity is less obvious. Let $\eta$ and $\eta'$ be paths in $\mathscr{P}$ and, for a smooth map $f \colon U \to \R$ on some D-open $U$, consider the function
\[
F' \colon (h,k) \mapsto f(\gamma + h\eta + k\eta')
\]
defined on $(-\varepsilon, \varepsilon) \times (-\varepsilon, \varepsilon)$ for a sufficiently small $\varepsilon > 0$ (there exists such an $\varepsilon$ because $U$ is D-open). This function $F'$ is smooth because it is the composite of $(h,k) \mapsto \gamma + h\eta + k\eta'$ with $f$, and the former is smooth because its adjoint $(-\varepsilon, \varepsilon) \times (-\varepsilon, \varepsilon) \times [a,b] \to \R^N : (h, k, t) \mapsto \gamma(t) + h\eta(t) + k\eta'(t)$ is smooth. Now, let us compute the directional derivatives $\nabla_{(1,0)}$, $\nabla_{(0,1)}$ and $\nabla_{(1,1)}$ of $F'$ at $(0,0)$. An easy check shows that:
\[
\nabla_{(1,0)} F' = \mathrm{D}(\eta)([f])
\hspace{1cm}
\text{and}
\hspace{1cm}
\nabla_{(0,1)} F' = \mathrm{D}(\eta')([f])
\]
Moreover, we can similarly see that the directional derivative along $(1,1)$ is
\[
\nabla_{(1,1)}F' = \mathrm{D}(\eta + \eta')([f])
\]
and so, using $\nabla_{(1,1)} = \nabla_{(1,0)} + \nabla_{(0,1)}$, we have
\[
\mathrm{D}(\eta + \eta')([f]) = \mathrm{D}(\eta)([f]) + \mathrm{D}(\eta')([f])
\]
as desired. \\

We have now demonstrated linearity of $\mathrm{D}$. It remains to show that $\mathrm{D}$ is bijective. First, to see that $\mathrm{D}$ is injective, suppose that $\mathrm{D}\eta = 0$. Consider the function
\[
f \colon \mathscr{P} \to \R : \alpha \mapsto \int_a^b \eta(t) \cdot \alpha(t)\,dt
\]
This is smooth as per Example \ref{ex:smooth_function_on_path_space}. Noting that
\begin{align*}
(\mathrm{D}\eta)([f]) &= \frac{d}{dh}\Big\vert_{h=0}\int_a^b \eta(t) \cdot (\gamma(t) + h\eta(t))\,dt \\
&= \int_a^b |\eta(t)|^2\,dt
\end{align*}
which, by assumption, is zero, $\eta$, being smooth, must itself be zero, as desired. \\

Finally, to see that $\mathrm{D}$ is surjective, let $v$ be a derivation. By representability, without loss of generality, we may suppose that there exists a plot $p \colon (-\varepsilon, \varepsilon) \to \mathscr{P}$ such that $v([f]) = \frac{d}{dh}\Big\vert_{h=0}f(p(h))$. Let $\widetilde{p}$ denote the smooth adjoint $\widetilde p \colon (-\varepsilon, \varepsilon) \times [a,b] \to \R^N : (h,t) \mapsto [p(h)](t)$. Fix some smooth $f \colon U \to \R$ defined on some $D$-open neighbourhood of  $\gamma$, and consider the following map:
\[
G' \colon (-\varepsilon, \varepsilon) \times (-\varepsilon, \varepsilon) \times [a,b] \to \R^N : (h,k,t) \mapsto \gamma(t) + h\int_0^1\partial_1 \widetilde{p}\left(ks,t\right)\,ds
\]
(The integral is computed component-wise.) As this is a smooth map, we have an associated adjoint plot $\widetilde{G'} \colon (-\varepsilon, \varepsilon) \times (-\varepsilon, \varepsilon) \to \mathscr{P}$. Now consider the directional derivatives $\nabla_{(1,0)}, \nabla_{(0,1)}$ and $\nabla_{(1,1)}$ of $f \circ \widetilde{G'}$ at $(0,0)$ (in defining the composite $f \circ \widetilde{G'}$, the domain of $\widetilde{G'}$ may need to be appropriately shrunk to a smaller open square about the origin; this is always possible because $U$ is D-open). Note that $\widetilde{G'}(0,0) = \gamma$. As $G'(0,k,t) = \gamma(t)$, we have that $\widetilde{G'}(0,k)$ is constant at $\gamma$ and so:
\[
\nabla_{(0,1)}(f \circ \widetilde{G'}) = 0
\]
Next, noting that
\begin{align*}
G'(h,0,t) &= \gamma(t) + h\int_0^1\partial_1 \widetilde{p}\left(0,t\right)\,ds \\
&= \gamma(t) + h\partial_1 \widetilde{p}\left(0,t\right)
\end{align*}
if we let $\eta$ denote the path $\partial_1 \widetilde{p}(0,-)$, we see that
\[
\nabla_{(1,0)}(f \circ \widetilde{G'}) = \lim_{h \to 0}\frac{f(\gamma + h\eta)-f(\gamma)}{h} = (\mathrm{D}\eta)([f])
\]
Moreover, noting that, for $h \neq 0$
\begin{align*}
G'(h,h,t) &= \gamma(t) + h\int_0^1\partial_1 \widetilde{p}\left(hs,t\right)\,ds \\
&= \gamma(t) + h\int_0^h\frac{1}{h}\partial_1 \widetilde{p}\left(u,t\right)\,du \\
&= \gamma(t) + \widetilde{p}(h,t) - \underbrace{\widetilde{p}(0,t)}_{= \gamma(t)} \\
&= \widetilde{p}(h,t)
\end{align*}
we have that $\widetilde{G'}(h,h) = p(h)$ and so:
\[
\nabla_{(1,1)}(f \circ \widetilde{G'}) = \lim_{h \to 0}\frac{f(p(h))-f(\gamma)}{h} = v([f])
\]
Finally, as $\nabla_{(1,1)} = \nabla_{(1,0)} + \nabla_{(0,1)}$, we have that $v = \mathrm{D}\eta$, and so $v$ lies in the image of $\mathrm{D}$, as desired.
\end{proof}

\begin{Remark}
In the calculus of variations, given a function $f \colon \mathscr{P}_{\mathbf{p}, \mathbf{q}} \to \R$, we consider perturbations, or variations, of an input path $\gamma$, by adding $h\eta$ for sufficiently small $h$ and a path $\eta$ which is zero at the endpoints, and then we take the derivative at $h=0$. We intuitively think of this as an infinitesimal perturbation in the direction of $\eta$. With the isomorphism $\mathrm{D}$ constructed above in Theorem \ref{thm:tangent_spaces_to_path_spaces}, we can now identify this infinitesimal perturbation as literally a tangent to the path space.
\end{Remark}

\section{The Euler-Lagrange Equations as the Gradient of the Action Functional}

We now consider gradients of smooth maps $X \to \R$ on diffeological spaces $X$.

\begin{Definition}[\textbf{Cotangent Spaces}]
Let $X$ be a diffeological space and let $x \in X$. The \emph{cotangent space to $X$ at $x$} is the vector space dual $(\mathrm{T}_xX)^*$.
\end{Definition}

\begin{Definition}[\textbf{Gradients}]
Let $X$ be a diffeological space, $f \colon X \to \R$ a smooth map and $x \in X$. We have an induced map $\mathrm{T}_xX \to \mathrm{T}_{f(x)}\R \cong \R$. This map, a cotangent at $x$, is the \textit{gradient of $f$ at $x$}, denoted $(\nabla f)(x)$.
\end{Definition}

In the case of $\R^n$, as a vector space, we have the musical isomorphism
\[
\sharp \colon \R^n \overset{\cong}\longrightarrow (\R^n)^*
\]
where $\mathbf{v} \in \R^n$ is sent to $\langle \mathbf{v}, -\rangle = \mathbf{v} \cdot (-)$. (More generally, we have such isomorphisms for the tangent spaces of Riemannian manifolds.) In the case of $\mathscr{P}$ and $\mathscr{P}_{\mathbf{0}, \mathbf{0}}$, we also have musical maps
\[
\sharp \colon \mathscr{P} \longrightarrow \mathscr{P}^*
\hspace{1cm}
\sharp \colon \mathscr{P}_{\mathbf{0}, \mathbf{0}} \longrightarrow \mathscr{P}_{\mathbf{0}, \mathbf{0}}^*
\]
where, in either case, a path $\alpha$ is sent to $\langle \alpha, - \rangle = \int \alpha \cdot (-)$. In these cases, $\sharp$ is no longer an isomorphism: it cannot be surjective as the dual space of an infinite dimensional vector space is always strictly larger than the original space. However, it is clearly still injective. Here, we are using the same notation $\sharp$ for all three maps, to emphasize the analogy that exists between them; the given context will always make it clear which exact map we intend. \\

In the case of $\R^n$, as a smooth manifold, at any given point $\mathbf{p}$, we have the isomorphism $\mathrm{D} \colon \R^n \to \mathrm{T}_{\mathbf{p}}\R^n$ as in Proposition \ref{prop:tangent_spaces_for_smooth_manifolds}, and the musical isomorphism is compatible with this map in the sense that we have the following commutative square:
\begin{center}
\begin{tikzpicture}[node distance = 2cm]
\node [] (A) {$\R^n$};
\node [right of = A] (B) {$(\R^n)^*$};
\node [below of = A] (C) {$\mathrm{T}_p\R^n$};
\node [below of = B] (D) {$(\mathrm{T}_p\R^n)^*$};

\draw [->] (A) -- (B) node[midway, anchor=south]{$\sharp$};
\draw [->] (A) -- (C) node[midway, anchor=east]{$\mathrm{D}$};
\draw [->] (D) -- (B) node[midway, anchor=west]{$\mathrm{D}^*$};
\draw [->] (C) -- (D) node[midway, anchor=south]{$\sharp$};
\end{tikzpicture}
\end{center}
The analogous pictures that we have for $\mathscr{P}$ and $\mathscr{P}_{\mathbf{p}, \mathbf{q}}$ are as follows:
\begin{center}
\begin{tikzpicture}[node distance = 2cm]
\node [] (A) {$\mathscr{P}$};
\node [right of = A] (B) {$(\mathscr{P})^*$};
\node [below of = A] (C) {$\mathrm{T}_\gamma\mathscr{P}$};
\node [below of = B] (D) {$(\mathrm{T}_\gamma\mathscr{P})^*$};

\draw [->] (A) -- (B) node[midway, anchor=south]{$\sharp$};
\draw [->] (A) -- (C) node[midway, anchor=east]{$\mathrm{D}$};
\draw [->] (D) -- (B) node[midway, anchor=west]{$\mathrm{D}^*$};
\draw [-, dashed] (C) -- (D);

\node [right of = A, xshift = 3cm] (AA) {$\mathscr{P}_{\mathbf{0}, \mathbf{0}}$};
\node [right of = AA] (BB) {$(\mathscr{P}_{\mathbf{0}, \mathbf{0}})^*$};
\node [below of = AA] (CC) {$\mathrm{T}_\gamma\mathscr{P}_{\mathbf{p}, \mathbf{q}}$};
\node [below of = BB] (DD) {$(\mathrm{T}_\gamma\mathscr{P}_{\mathbf{p}, \mathbf{q}})^*$};

\draw [->] (AA) -- (BB) node[midway, anchor=south]{$\sharp$};
\draw [->] (AA) -- (CC) node[midway, anchor=east]{$\mathrm{D}$};
\draw [->] (DD) -- (BB) node[midway, anchor=west]{$\mathrm{D}^*$};
\draw [-, dashed] (CC) -- (DD);
\end{tikzpicture}
\end{center}
Here, the maps $\mathrm{D}$ are those in Theorem \ref{thm:tangent_spaces_to_path_spaces}, and the bottom arrows are dashed to indicate that these are missing in the case of the path spaces (we will find that, for our purposes, these maps are not necessary). \\

Now, recall from earlier that we can think of the inclusion
\[
\mathscr{P}_{\mathbf{p}, \mathbf{q}} \hookrightarrow \mathscr{P}
\]
as an infinte dimensional variant of the inclusion
\[
i \colon H \hookrightarrow \R^{n+k}
\]
where $n, k > 0$, $H$ is the affine subspace $\{(x_1,\dots,x_{n+k}) \mid x_{n+1} = c_{n+1}, \dots, x_{n+k} = c_{n+k}\}$ for some constants $c_{n+1}, \dots, c_{n+k}$, and the inclusion is the identity. The tangent space to $H$, at any point, is $H_0 := \{(x_1,\dots,x_{n+k}) \mid x_{n+1} = 0, \dots, x_{n+k} = 0\}$, and, for any point $\mathbf{p}$, we have a map
\[
\mathrm{D} \colon H_0 \to \mathrm{T}_{\mathbf{p}}H
\]
defined by the formula which we have used for the other analogous maps, $\mathrm{D}(\mathbf{v})([f]) = \frac{d}{dh}\Big\vert_{h=0}f(\mathbf{p} + h\mathbf{v})$. Letting $i_0$ denote the inclusion $H_0 \to \R^{n+k}$, combining the $\mathrm{D}$ and $\sharp$ maps, together with dualization, we have the following commutative diagram:
\begin{center}
\begin{tikzpicture}[node distance = 2cm]
\node at (0,0) (A) {$H_0$};
\node at (0,-3) (B) {$\R^{n+k}$};
\node at (2,2) (C) {$(H_0)^*$};
\node at (2,-1) (D) {$(\R^{n+k})^*$};

\node at (4,0) (AA) {$\mathrm{T}_{\mathbf{p}}H$};
\node at (4,-3) (BB) {$\mathrm{T}_{\mathbf{p}}\R^{n+k}$};
\node at (6,2) (CC) {$(\mathrm{T}_{\mathbf{p}}H)^*$};
\node at (6,-1) (DD) {$(\mathrm{T}_{\mathbf{p}}\R^{n+k})^*$};

\draw [right hook->] (A) -- (B) node[midway, anchor=east]{$i_0$};
\draw [-] (2,-0.7) -- (2,-0.1);
\draw [->] (2,0.1) -- (2,1.6) node[midway, anchor=east, yshift=-2mm]{$i_0^*$};
\draw [->] (A) -- (C) node[midway, anchor=south]{$\sharp$};
\draw [->] (B) -- (D) node[midway, anchor=south]{$\sharp$};

\draw [->] (AA) -- (BB) node[midway, anchor=east, yshift=-3mm]{$\mathrm{T}_{\mathbf{p}}i$};
\draw [->] (DD) -- (CC) node[midway, anchor=east, yshift=-6mm]{$(\mathrm{T}_{\mathbf{p}}i)^*$};
\draw [->] (AA) -- (CC) node[midway, anchor=south]{$\sharp$};
\draw [->] (BB) -- (DD) node[midway, anchor=south]{$\sharp$};

\draw [->] (A) -- (AA) node[midway, anchor=south, xshift=7mm]{$\mathrm{D}$};
\draw [->] (B) -- (BB) node[midway, anchor=south]{$\mathrm{D}$};
\draw [->] (CC) -- (C) node[midway, anchor=south]{$\mathrm{D}^*$};
\draw [-] (5,-1) -- (4.1,-1);
\draw [->] (3.9,-1) -- (2.8,-1) node[midway,anchor=south, xshift=1mm]{$\mathrm{D}^*$};
\end{tikzpicture}
\end{center}
Let
\[
f \colon \R^{n+k} \to \R : (x_1, \dots, x_{n+k}) \mapsto f(x_1,\dots,x_{n+k})
\]
be a smooth function, and let
\[
fi \colon H \to \R : (x_1,\dots,x_n,c_{n+1},\dots,c_{n+k}) \mapsto f(x_1,\dots,x_n,c_{n+1},\dots,c_{n+k})
\]
be the restriction of $f$ to $H$. For a given $\mathbf{p} \in H$, the definition of $(\nabla fi)(\mathbf{p})$ in the sense of diffeological spaces (or smooth manifolds) and the vector
\[
\left(\frac{\partial f}{\partial x_1}(i(\mathbf{p})), \dots, \frac{\partial f}{\partial x_n}(i(\mathbf{p}))\right)
\]
that one typically uses in multivariable calculus, correspond in the following sense:
\begin{center}
\begin{tikzpicture}[node distance = 2cm]
\node at (0,0) (A) {$\color{gray} H_0$};
\node at (0,-3) (B) {$\color{gray} \R^{n+k}$};
\node at (2,2) (C) {$\color{gray} (H_0)^*$};
\node at (2,-1) (D) {$\color{gray} (\R^{n+k})^*$};

\node at (4,0) (AA) {$\color{gray} \mathrm{T}_{\mathbf{p}}H$};
\node at (4,-3) (BB) {$\color{gray} \mathrm{T}_{\mathbf{p}}\R^{n+k}$};
\node at (6,2) (CC) {$\color{gray} (\mathrm{T}_{\mathbf{p}}H)^*$};
\node at (6,-1) (DD) {$\color{gray} (\mathrm{T}_{\mathbf{p}}\R^{n+k})^*$};

\draw [right hook->] (A) -- (B) node[midway, anchor=east]{$\color{gray} i_0$};
\draw [-] (2,-0.7) -- (2,-0.1);
\draw [->] (2,0.1) -- (2,1.6) node[midway, anchor=east, yshift=-2mm]{$\color{gray} i_0^*$};
\draw [->] (A) -- (C) node[midway, anchor=south]{$\color{gray} \sharp$};
\draw [->] (B) -- (D) node[midway, anchor=south]{$\color{gray} \sharp$};

\draw [->] (AA) -- (BB) node[midway, anchor=east, yshift=-3mm]{$\color{gray} \mathrm{T}_{\mathbf{p}}i$};
\draw [->] (DD) -- (CC) node[midway, anchor=east, yshift=-6mm]{$\color{gray} (\mathrm{T}_{\mathbf{p}}i)^*$};
\draw [->] (AA) -- (CC) node[midway, anchor=south]{$\color{gray} \sharp$};
\draw [->] (BB) -- (DD) node[midway, anchor=south]{$\color{gray} \sharp$};

\draw [->] (A) -- (AA) node[midway, anchor=south, xshift=7mm]{$\color{gray} \mathrm{D}$};
\draw [->] (B) -- (BB) node[midway, anchor=south]{$\color{gray} \mathrm{D}$};
\draw [->] (CC) -- (C) node[midway, anchor=south]{$\color{gray} \mathrm{D}^*$};
\draw [-] (5,-1) -- (4.1,-1);
\draw [->] (3.9,-1) -- (2.8,-1) node[midway,anchor=south, xshift=1mm]{$\color{gray} \mathrm{D}^*$};

\node at (7.75,-1) (E) {$(\nabla f)(\mathbf{p})$};
\node at (7.75,2.5) (F) {$(\nabla fi)(\mathbf{p})$};
\node at (-2.25,2.5) (G) {$[\frac{\partial f}{\partial x_1}(i(\mathbf{p})) \,\, \cdots \,\, \frac{\partial f}{\partial x_{n+k}}(i(\mathbf{p}))] = \mathrm{D}^*[(\nabla fi)(\mathbf{p})]$};

\node at (-4.25,-3) (H) {$\left[\begin{array}{c} \frac{\partial f}{\partial x_1}(i(\mathbf{p})) \\ \vdots \\ \frac{\partial f}{\partial x_{n+k}}(i(\mathbf{p})) \end{array}\right]$};
\node at (-2.25,-1) (I) {$[\frac{\partial f}{\partial x_1}(i(\mathbf{p})) \,\, \cdots \,\, \frac{\partial f}{\partial x_{n+k}}(i(\mathbf{p}))]$};

\draw [|->] (E) -- (F);
\draw [|->] (F) -- (G);
\draw [|->] (H) -- (I);
\draw [|->] (I) -- (G);
\end{tikzpicture}
\end{center}
(At the top-left corner, on $H_0$, the transformation $[\partial_1f(i(\mathbf{p})) \,\, \cdots \,\, \partial_{n+k}f(i(\mathbf{p}))]$ is of course the same as the transformation $[\partial_1f(i(\mathbf{p})) \,\, \cdots \,\, \partial_{n}f(i(\mathbf{p})) \, 0 \cdots 0]$, as the final $k$ components of vectors in $H_0$ are always zero.) \\

Now let us consider the case of our path spaces. In this case, we let $i$ denote the inclusion $\mathscr{P}_{\mathbf{p}, \mathbf{q}} \to \mathscr{P}$ and $i_0$ the inclusion $\mathscr{P}_{\mathbf{0}, \mathbf{0}} \to \mathscr{P}$. We then get the following analogous diagram:
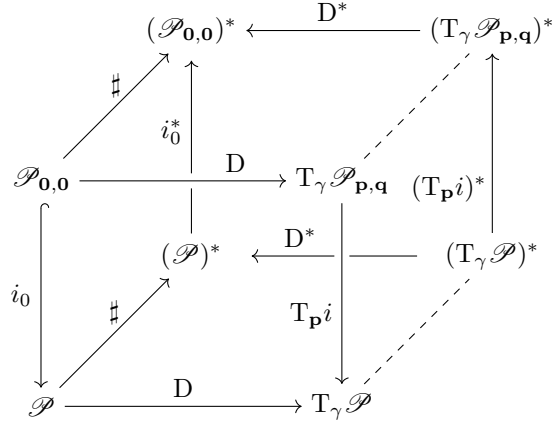
\begin{center}
\begin{tikzpicture}[node distance = 2cm]
\node at (0,0) (A) {$\mathscr{P}_{\mathbf{0},\mathbf{0}}$};
\node at (0,-3) (B) {$\mathscr{P}$};
\node at (2,2) (C) {$(\mathscr{P}_{\mathbf{0},\mathbf{0}})^*$};
\node at (2,-1) (D) {$(\mathscr{P})^*$};

\node at (4,0) (AA) {$\mathrm{T}_{\gamma}\mathscr{P}_{\mathbf{p}, \mathbf{q}}$};
\node at (4,-3) (BB) {$\mathrm{T}_{\gamma}\mathscr{P}$};
\node at (6,2) (CC) {$(\mathrm{T}_{\gamma}\mathscr{P}_{\mathbf{p}, \mathbf{q}})^*$};
\node at (6,-1) (DD) {$(\mathrm{T}_{\gamma}\mathscr{P})^*$};

\draw [right hook->] (A) -- (B) node[midway, anchor=east]{$i_0$};
\draw [-] (2,-0.7) -- (2,-0.1);
\draw [->] (2,0.1) -- (2,1.6) node[midway, anchor=east, yshift=-2mm]{$i_0^*$};
\draw [->] (A) -- (C) node[midway, anchor=south]{$\sharp$};
\draw [->] (B) -- (D) node[midway, anchor=south]{$\sharp$};

\draw [->] (AA) -- (BB) node[midway, anchor=east, yshift=-3mm]{$\mathrm{T}_{\mathbf{p}}i$};
\draw [->] (DD) -- (CC) node[midway, anchor=east, yshift=-6mm, xshift=0.5mm]{$(\mathrm{T}_{\mathbf{p}}i)^*$};
\draw [-, dashed] (AA) -- (CC) node[midway, anchor=south]{};
\draw [-, dashed] (BB) -- (DD) node[midway, anchor=south]{};

\draw [->] (A) -- (AA) node[midway, anchor=south, xshift=7mm]{$\mathrm{D}$};
\draw [->] (B) -- (BB) node[midway, anchor=south]{$\mathrm{D}$};
\draw [->] (CC) -- (C) node[midway, anchor=south]{$\mathrm{D}^*$};
\draw [-] (5,-1) -- (4.1,-1);
\draw [->] (3.9,-1) -- (2.8,-1) node[midway,anchor=south, xshift=1mm]{$\mathrm{D}^*$};
\end{tikzpicture}
\end{center}

\begin{Remark}
\label{rmk:i0star_after_sharp_injective}
One difference between this picture and the analogous one for $H$ and $\R^{n+k}$ above, is that $i_0^* \circ \sharp$, as a map $\mathscr{P} \to (\mathscr{P}_{\mathbf{0}, \mathbf{0}})^*$, is injective. This is because, given $\gamma \in \mathscr{P}$, if $\int \gamma \cdot \eta$ is zero for all $\eta$ which are zero at the endpoints, then, by a simple continuity argument, $\gamma$ itself must be zero. 
\end{Remark}

Consider now a smooth function
\[
f \colon \mathscr{P} \to \R
\]
and the restriction $fi \colon \mathscr{P}_{\mathbf{p}, \mathbf{q}} \to \R$. In analogy with the case of $H \to \R^{n+k}$ above, we wish to compute $\mathrm{D}^*[(\nabla fi)(\gamma)]$ in terms of a tangent vector in $\mathscr{P}$. For a general $f$, there is no guarantee of such a computation. However, we have the following for a class of functions $f$, namely those which come from a Lagrangian.

\begin{Theorem}
\label{thm:euler_lagrange_equations_as_gradients}
Let $L(x_1,\dots,x_n,\dot{x}_1,\dots,\dot{x}_n,t)$ be a smooth function $\R^{2N+1} \to \R$ and consider the following smooth function on paths
\[
S \colon \mathscr{P} \to \R : \gamma \mapsto \int_a^b L(\gamma(t), \gamma'(t), t)\,dt
\]
as well as the restriction $Si$ to $\mathscr{P}_{\mathbf{p}, \mathbf{q}}$. For any $\gamma \in \mathscr{P}_{\mathbf{p}, \mathbf{q}}$, let $\emph{EL}_\gamma \in \mathscr{P}$, which we shall refer to as the Euler-Lagrange path, be the following path:
\[
\mathrm{EL}_\gamma(t) = \left(\frac{\partial L}{\partial x_1} - \frac{d}{dt}\frac{\partial L}{\partial \dot{x}_1}, \cdots, \frac{\partial L}{\partial x_N} - \frac{d}{dt}\frac{\partial L}{\partial \dot{x}_N}\right)
\]
Then $\mathrm{EL}_\gamma$ and $(\nabla Si)(\gamma)$ are related as follows:
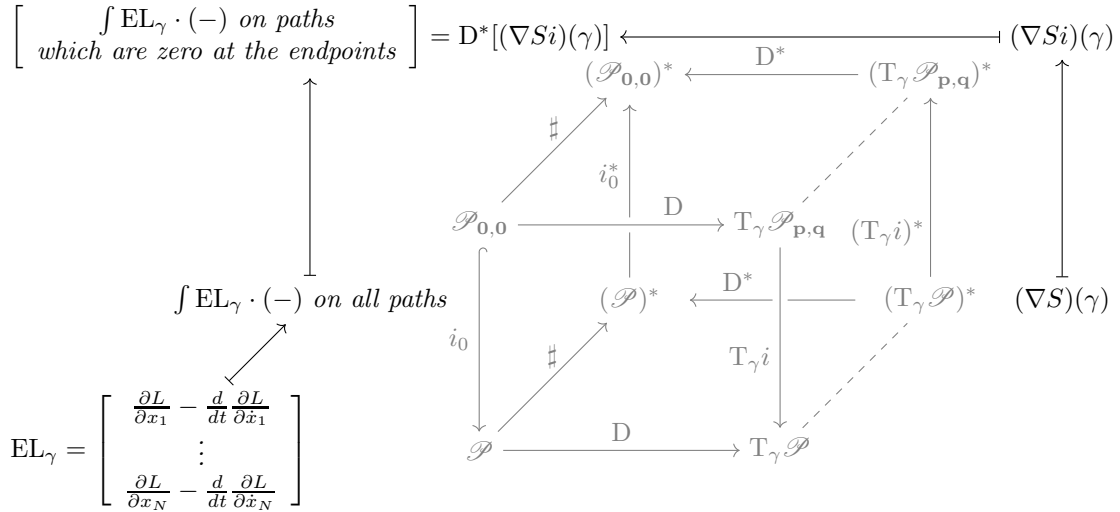
\begin{center}
\begin{tikzpicture}[node distance = 2cm]
\node at (0,0) (A) {$\color{gray} \mathscr{P}_{\mathbf{0},\mathbf{0}}$};
\node at (0,-3) (B) {$\color{gray} \mathscr{P}$};
\node at (2,2) (C) {$\color{gray} (\mathscr{P}_{\mathbf{0},\mathbf{0}})^*$};
\node at (2,-1) (D) {$\color{gray} (\mathscr{P})^*$};

\node at (4,0) (AA) {$\color{gray} \mathrm{T}_{\gamma}\mathscr{P}_{\mathbf{p}, \mathbf{q}}$};
\node at (4,-3) (BB) {$\color{gray} \mathrm{T}_{\gamma}\mathscr{P}$};
\node at (6,2) (CC) {$\color{gray} (\mathrm{T}_{\gamma}\mathscr{P}_{\mathbf{p}, \mathbf{q}})^*$};
\node at (6,-1) (DD) {$\color{gray} (\mathrm{T}_{\gamma}\mathscr{P})^*$};

\draw [right hook->, gray] (A) -- (B) node[midway, anchor=east]{$\color{gray} i_0$};
\draw [-, gray] (2,-0.7) -- (2,-0.1);
\draw [->, gray] (2,0.1) -- (2,1.6) node[midway, anchor=east, yshift=-2mm]{$\color{gray} i_0^*$};
\draw [->, gray] (A) -- (C) node[midway, anchor=south]{$\color{gray} \sharp$};
\draw [->, gray] (B) -- (D) node[midway, anchor=south]{$\color{gray} \sharp$};

\draw [->, gray] (AA) -- (BB) node[midway, anchor=east, yshift=-3mm]{$\color{gray} \mathrm{T}_{\gamma}i$};
\draw [->, gray] (DD) -- (CC) node[midway, anchor=east, yshift=-6mm, xshift=0.5mm]{$\color{gray} (\mathrm{T}_{\gamma}i)^*$};
\draw [-, dashed, gray] (AA) -- (CC) node[midway, anchor=south]{};
\draw [-, dashed, gray] (BB) -- (DD) node[midway, anchor=south]{};

\draw [->, gray] (A) -- (AA) node[midway, anchor=south, xshift=7mm]{$\color{gray} \mathrm{D}$};
\draw [->, gray] (B) -- (BB) node[midway, anchor=south]{$\color{gray} \mathrm{D}$};
\draw [->, gray] (CC) -- (C) node[midway, anchor=south]{$\color{gray} \mathrm{D}^*$};
\draw [-, gray] (5,-1) -- (4.1,-1);
\draw [->, gray] (3.9,-1) -- (2.8,-1) node[midway,anchor=south, xshift=1mm]{$\color{gray} \mathrm{D}^*$};

\node at (7.75,-1) (E) {$(\nabla S)(\gamma)$};
\node at (7.75,2.5) (F) {$(\nabla Si)(\gamma)$};
\node at (-2.25,2.5) (G) {$\left[\begin{array}{c}\int \mathrm{EL}_\gamma \cdot (-) \,\, \text{on paths} \\ \text{which are zero at the endpoints}\end{array}\right] = \mathrm{D}^*[(\nabla Si)(\gamma)]$};

\node at (-4.25,-3) (H) {$\mathrm{EL}_\gamma = \left[\begin{array}{c} \frac{\partial L}{\partial x_1} - \frac{d}{dt}\frac{\partial L}{\partial \dot{x}_1} \\ \vdots \\ \frac{\partial L}{\partial x_N} - \frac{d}{dt}\frac{\partial L}{\partial \dot{x}_N} \end{array}\right]$};
\node at (-2.25,-1) (I) {$\int \mathrm{EL}_\gamma \cdot (-) \,\, \text{on all paths}$};

\draw [|->] (E) -- (F);
\draw [|->] (F) -- (G);
\draw [|->] (H) -- (I);
\draw [|->] (I) -- (G);
\end{tikzpicture}
\end{center}
In this sense, the Euler-Lagrange path ``is'' the gradient of the action functional on $\mathscr{P}_{\mathbf{p}, \mathbf{q}}$. Moreover, $(\nabla Si)(\gamma)$ is zero exactly when $\emph{EL}_\gamma$ is zero; thus, in solving the Euler-Lagrange equations, one is precisely solving for the zeros of the gradient of the action functional.
\end{Theorem}

\begin{proof}
By construction, for $\alpha \in \mathscr{P}_{\mathbf{0}, \mathbf{0}}$, we have:
\begin{align*}
\mathrm{D}^*[(\nabla Si)(\gamma)](\alpha) &= [(\nabla Si)(\gamma)](\mathrm{D}\alpha) \\
&= [(\nabla S)(\gamma)](\mathrm{T}_{\gamma}i(\mathrm{D}\alpha)) \\
&= (\mathrm{T}_{\gamma}i(\mathrm{D}\alpha))(S) \\
&= (\mathrm{D}\alpha)(Si) \\
&= \frac{d}{dh}\Big\vert_{h=0}\int_a^b L(\gamma(t) + h\alpha(t), \gamma'(t) + h\alpha'(t), t)\,dt
\end{align*}
For the derivative, by differentiating under the integral sign and integrating by parts, we have:
\begin{align*}
\frac{d}{dh}\int_a^b L(\gamma(t) + h\alpha(t), \gamma'(t) + h\alpha'(t), t)\,dt &= \sum_{i=1}^N\left(\int_a^b \frac{\partial L}{\partial x_i}\alpha_i(t) + \frac{\partial L}{\partial \dot{x}_i}\alpha_i'(t)\,dt\right) \\
&= \sum_{i=1}^N\int_a^b \frac{\partial L}{\partial x_i}\alpha_i(t)\,dt + \sum_{i=1}^N\int_a^b\frac{\partial L}{\partial \dot{x}_i}\alpha_i'(t)\,dt \\
&= \sum_{i=1}^N\int_a^b \frac{\partial L}{\partial x_i}\alpha_i(t)\,dt \\
&\hspace{10mm} + \sum_{i=1}^N\left(-\int_a^b\frac{d}{dt}\frac{\partial L}{\partial \dot{x}_i}\alpha_i(t)\,dt + \underbrace{\left[\frac{\partial L}{\partial \dot{x}_i}\alpha_i(t)\right]_a^b}_{\alpha_i(a) = \alpha_i(b) = 0}\right) \\
&= \sum_{i=1}^N\int_a^b\left(\frac{\partial L}{\partial x_i} - \frac{d}{dt}\frac{\partial L}{\partial \dot{x}_i}\right)\alpha_i(t)\,dt
\end{align*}
Here, in the final sum, the terms $\frac{\partial L}{\partial x_i}$ and $\frac{\partial L}{\partial \dot{x}_i}$ are evaluated at $(\gamma + h\alpha, \gamma' + h\alpha', t)$; and so upon evaluation of the derivative at $h=0$, we find that
\[
\mathrm{D}^*[(\nabla Si)(\gamma)](\alpha) = \int_a^b \mathrm{EL}_\gamma(t) \cdot \alpha(t)\,dt = \langle \mathrm{EL}_\gamma, \alpha \rangle
\]
as desired. Finally, to see $(\nabla Si)(\gamma)$ is zero exactly when $\mathrm{EL}_\gamma$ is zero, note that $\mathrm{D}^*$, on $(\mathrm{T}_\gamma\mathscr{P}_{\mathbf{p}, \mathbf{q}})^*$, is an isomorphism by Theorem \ref{thm:tangent_spaces_to_path_spaces} and $i_0^* \circ \sharp$, on $\mathscr{P}$, is injective as in Remark \ref{rmk:i0star_after_sharp_injective}. 
\end{proof}

\begin{Remark}
The calculation in the proof above for $\frac{d}{dh}\int_a^b L(\gamma(t) + h\alpha(t), \gamma'(t) + h\alpha'(t), t)\,dt$ is a standard one in the calculus of variations, but here its final outcome appears in a new light. We see that the resulting Euler-Lagrange equations are the components of a vector $\mathrm{EL}_\gamma$ tangent to the path space $\mathscr{P}$, and that this vector is precisely the gradient of the action functional $Si$.
\end{Remark}

Armed with the geometric interpretation of the Euler-Lagrange equations as the gradient of the action functional, we may also similarly interpret the constrained Euler-Lagrange equations. Recall first the Lagrange multiplier theorem from ordinary multivariable calculus. This says that, given a smooth function $f \colon \R^N \to \R$, another smooth function $g \colon \R^N \to \R$, and a regular value $c$ of $g$, for $\mathbf{x}$ to be a stationary point of the restriction $T \to \R^N \to \R$, where $T$ is the level set of $g$ at $c$, it is necessary and sufficient for $(\nabla f)(\mathbf{x})$ to be parallel to $(\nabla g)(\mathbf{x})$. One can interpret this geometrically as follows: if $\nabla f$ has a non-zero component in a direction parallel to $T$ (this is equivalent to $\nabla f$ not being parallel to $\nabla g$), then there is a direction, parallel to $T$, in which we can perturb the input and increase or decrease the value of the output. \\

We can consider an analogous problem for $\mathscr{P}_{\mathbf{p}, \mathbf{q}}$. In fact, we can consider two different kinds of constraints in this case:
\begin{itemize}
	\item[(i)] Given a smooth function $g \colon \R^N \to \R$ and a regular value $c$ of $g$, we wish to locally minimize or maximize the action functional $Si$ of Theorem \ref{thm:euler_lagrange_equations_as_gradients} subject to the constraint the $\gamma$ must map into $T$, the level set of $g$ at $c$. Note that this is not quite the same as the finite dimensional problem above, in that we are not using the level set of a smooth function out of $\mathscr{P}_{\mathbf{p}, \mathbf{q}}$.
	\item[(ii)] Given a smooth function $M \colon \R^{2N+1} \to \R$, locally minimize or maximize the action function $Si$ of Theorem \ref{thm:euler_lagrange_equations_as_gradients} subject to the constraint that $\int_a^b M(\gamma(t), \gamma'(t), t)\,dt = c$ for some constant $c$. In this case the constraint function is indeed one on the path space $\mathscr{P}_{\mathbf{p}, \mathbf{q}}$.
\end{itemize}
For (i), as is standard (see \cite{GelfandFomin}), by considering perturbations of the form $\gamma + h\eta$, one finds that, if $Si$ is stationary on $T$ at $\gamma$, then there exists some smooth $\lambda \colon [a,b] \to \R$ such that:
\begin{align*}
\frac{\partial L}{\partial x_1} + \lambda\frac{\partial g}{\partial x_1} &= \frac{d}{dt}\frac{\partial L}{\partial \dot{x}_1} \\
&\vdots \\
\frac{\partial L}{\partial x_N} + \lambda\frac{\partial g}{\partial x_N} &= \frac{d}{dt}\frac{\partial L}{\partial \dot{x}_N}
\end{align*}
With our terminology from Theorem \ref{thm:euler_lagrange_equations_as_gradients} above, we see that this amounts exactly to that $\mathrm{EL}_\gamma$, which we can think of as $(\nabla Si)(\gamma)$, is parallel to $(\nabla g) \circ \gamma$, for each $t \in [a,b]$, with a smoothly varying constant of proportionality. We can interpret this geometrically in much the same manner as we did for the finite dimensional problem above: if the gradient of the functional has a non-zero component parallel to $T$, at any point along the input path, we can increase or decrease the action by perturbing the path, while remaining within $T$, at the corresponding value of $t \in [a,b]$. \\

On the other hand, for (ii), again, as is standard (see \cite{GelfandFomin}), one finds that the stationary points are found by solving the equations, for $\lambda \in \R$:
\begin{align*}
\frac{\partial L}{\partial x_1} - \frac{d}{dt}\frac{\partial L}{\partial \dot{x}_1} + &\lambda\left(\frac{\partial M}{\partial x_1} - \frac{d}{dt}\frac{\partial M}{\partial \dot{x}_1}\right) = 0 \\
&\vdots \\
\frac{\partial L}{\partial x_N} - \frac{d}{dt}\frac{\partial L}{\partial \dot{x}_N} + &\lambda\left(\frac{\partial M}{\partial x_N} - \frac{d}{dt}\frac{\partial M}{\partial \dot{x}_N}\right) = 0
\end{align*}
With our terminology from Theorem \ref{thm:euler_lagrange_equations_as_gradients} above, we see that this amounts exactly to that $\mathrm{EL}_\gamma^L$, the Euler-Lagrange path corresponding to $L$, which we can think of as $(\nabla Si)(\gamma)$, is parallel to $\mathrm{EL}_\gamma^M$, the Euler-Lagrange path corresponding to $M$, which we can think of as the gradient of the constraint function.

\section{Examples from Geometry, Mechanics and Machine Learning}

We now provide several examples from various fields to illustrate the general theory above. These examples show clearly that the Euler-Lagrange paths behave, with respect to stationary paths, in exactly the way that gradients do in ordinary multivariable calculus with respect to stationary points. For example, if there is a local minimum at a path $\gamma$, then $-\mathrm{EL}_\gamma$, representing the negative of the gradient, always points toward $\gamma$.

\begin{Example}[\textbf{Euclidean geometry}]
In the Euclidean plane $\E^2$, a path $\gamma(t) = (x(t), y(t))$, for $t \in [a,b]$, is a geodesic exactly when
\[
\ell_{\text{euc}}(\gamma) = \int_a^b \sqrt{x'(t)^2 + y'(t)^2}\,dt
\]
is minimal. Let us take $a = 0$ and $b = 1$ and let us take the initial and final points to be $(0,0)$ and $(1,0)$, respectively. The length-minimizing path, the geodesic, is then of course $\gamma_0(t) = (t,0)$. As $\gamma_0$ is a minimizing path, at other paths, we expect the negative gradient $-\nabla \ell_{\text{euc}}$, represented by the Euler-Lagrange path, to point toward $\gamma_0$. Consider the paths $\gamma_1(t) = (t, t(1-t))$ and $\gamma_2(t) = (t, t(t-1))$. The corresponding Euler-Lagrange paths are easily calculated to be:
\[
\mathrm{EL}_{\gamma_1}(t) = \left( \frac{2(2t-1)}{(1+(1-2t)^2)^{3/2}}, \frac{2}{(1+(1-2t)^2)^{3/2}} \right)
\]
\[
\mathrm{EL}_{\gamma_2}(t) = \left( \frac{2(2t-1)}{(1+(1-2t)^2)^{3/2}}, \frac{-2}{(1+(1-2t)^2)^{3/2}} \right)
\]
As in Figure \ref{fig:euc_geom}, we see that $-\mathrm{EL}_{\gamma_1}$ and $-\mathrm{EL}_{\gamma_2}$ point in the expected directions, toward the minimizing path.
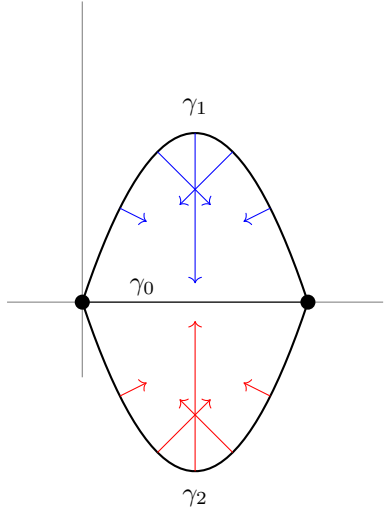
\begin{figure}[h]
\begin{center}
\begin{tikzpicture}
\draw [-, gray] (0,-1) -- (0,4);
\draw [-, gray] (-1,0) -- (4,0);

\node[circle, minimum size=2mm, inner sep=0pt, fill=black] (A) at (0,0) {};
\node[circle, minimum size=2mm, inner sep=0pt, fill=black] (A) at (3,0) {};

\draw [black, thick,  domain=0:3, samples=40] plot ({\x}, {\x*(3-\x)} );
\draw [black, thick,  domain=0:3, samples=40] plot ({\x}, {\x*(\x-3)} );
\draw [-] (0,0) -- (3,0);

\node at (1.5,2.6) (A) {$\gamma_1$};
\node at (1.5,-2.6) (B) {$\gamma_2$};
\node at (0.8, 0.2) (C) {$\gamma_0$};

\draw [->, blue] ( 0.5 , 1.25 ) -- ( 0.8577708763999663 , 1.0711145618000169 );
\draw [->, blue] ( 1 , 2 ) -- ( 1.7071067811865475 , 1.2928932188134525 );
\draw [->, blue] ( 1.5 , 2.25 ) -- ( 1.5 , 0.25 );
\draw [->, blue] ( 2 , 2 ) -- ( 1.2928932188134525 , 1.2928932188134525 );
\draw [->, blue] ( 2.5 , 1.25 ) -- ( 2.1422291236000337 , 1.0711145618000169 );

\draw [->, red] ( 0.5 , -1.25 ) -- ( 0.8577708763999663 , -1.0711145618000169 );
\draw [->, red] ( 1 , -2 ) -- ( 1.7071067811865475 , -1.2928932188134525 );
\draw [->, red] ( 1.5 , -2.25 ) -- ( 1.5 , -0.25 );
\draw [->, red] ( 2 , -2 ) -- ( 1.2928932188134525 , -1.2928932188134525 );
\draw [->, red] ( 2.5 , -1.25 ) -- ( 2.1422291236000337 , -1.0711145618000169 );
\end{tikzpicture}
\end{center}
\caption{The blue vectors are the values of $-\mathrm{EL}_{\gamma_1}(t)$; the red vectors are the values of $-\mathrm{EL}_{\gamma_2}(t)$.}
\label{fig:euc_geom}
\end{figure}
\end{Example}

\begin{Example}[\textbf{Hyperbolic geometry}]
In the hyperbolic plane $\mathbb{H}^2 = \{(x,y) \mid y > 0\}$, taken to be the upper half-plane, a path $\gamma(t) = (x(t), y(t))$, for $t \in [a,b]$, is a geodesic exactly when
\[
\ell_{\text{hyp}}(\gamma) = \int_a^b \frac{\sqrt{x'(t)^2 + y'(t)^2}}{y(t)}\,dt
\]
is minimal. Let us take $a = \frac{\pi}{4}$ and $b = \frac{3\pi}{4}$ and let us take the initial and final points to be $(-1, 1)$ and $(1, 1)$, respectively. In the hyperbolic plane, the geodesics are the vertical lines and the circular arcs centred on the horizontal axis. Thus, the minimizing path in the case at hand is $\gamma_0(t) = (\sqrt{2}\cos(t), \sqrt{2}\sin(t))$. As $\gamma_0$ is a minimizing path, at other paths, we expect the negative gradient $-\nabla \ell_{\text{hyp}}$, represented by the Euler-Lagrange path, to point toward $\gamma_0$. Consider the straight-line path $\gamma_1(t) = (\frac{4t}{\pi}-2, 1)$. The corresponding Euler-Lagrange path is as follows:
\[
\mathrm{EL}_{\gamma_1}(t) = \left( 0, -\frac{4}{\pi} \right)
\]
As in Figure \ref{fig:hyp_geom}, we see that $-\mathrm{EL}_{\gamma_1}$ points in the expected direction, toward the minimizing path.
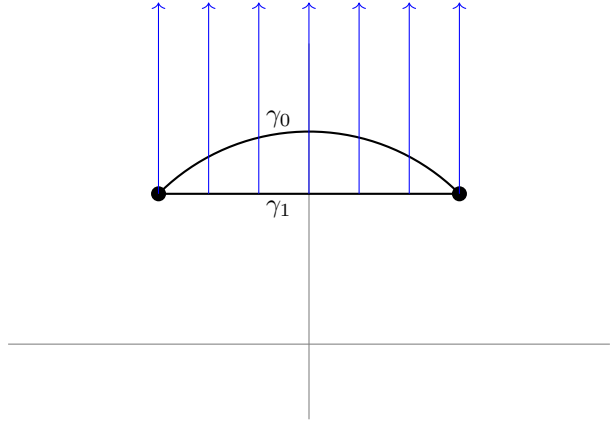
\begin{figure}[h]
\begin{center}
\begin{tikzpicture}
\draw [-, gray] (0,-1) -- (0,4);
\draw [-, gray] (-4,0) -- (4,0);

\node[circle, minimum size=2mm, inner sep=0pt, fill=black] (A) at (-2,2) {};
\node[circle, minimum size=2mm, inner sep=0pt, fill=black] (A) at (2,2) {};

\draw [black, thick,  domain=0.785:2.356, samples=40] plot ({2*1.4142*cos(deg(\x))}, {2*1.4142*sin(deg(\x))} );
\draw [, thick,  domain=0.785:2.356, samples=40] plot ({2*4*\x/3.14159 - 2*2}, {2*1} );

\node at (-0.4,3) (A) {$\gamma_0$};
\node at (-0.4,1.8) (B) {$\gamma_1$};

\draw [->, blue] ( -2.001012226293055 , 2 ) -- ( -2.001012226293055 , 4.546481240391012 );
\draw [->, blue] ( -1.3342585548506753 , 2 ) -- ( -1.3342585548506753 , 4.546481240391012 );
\draw [->, blue] ( -0.6675048834082955 , 2 ) -- ( -0.6675048834082955 , 4.546481240391012 );
\draw [->, blue] ( -0.0007512119659152106 , 2 ) -- ( -0.0007512119659152106 , 4.546481240391012 );
\draw [->, blue] ( 0.6660024594764646 , 2 ) -- ( 0.6660024594764646 , 4.546481240391012 );
\draw [->, blue] ( 1.3327561309188436 , 2 ) -- ( 1.3327561309188436 , 4.546481240391012 );
\draw [->, blue] ( 1.9995098023612243 , 2 ) -- ( 1.9995098023612243 , 4.546481240391012 );
\end{tikzpicture}
\end{center}
\caption{The blue vectors are the values of $-\mathrm{EL}_{\gamma_1}(t)$.}
\label{fig:hyp_geom}
\end{figure}
\end{Example}

\begin{Example}[\textbf{Spherical geometry}]
On the sphere $\Sph^2 = \{(x,y,z) \mid x^2+ y^2 + z^2 = 1\}$, we can find geodesics by finding paths which minimize the 3-dimensional Euclidean length
\[
\int_a^b\sqrt{x'(t)^2+y'(t)^2+z'(t)^2}\,dt
\] subject to the restriction that $\gamma$ lie on $\Sph^2$. This is an example with a constraint. The constraint is that $g(x,y,z) = 1$, where $g(x,y,z) = x^2+y^2+z^2$. An example geodesic is $\gamma(t) = (0, \cos(t), \sin(t))$ for $t \in [0, \pi]$, part of the great circle contained in the $yz$-plane. We find that:
\[
\mathrm{EL}_\gamma(t) = (0,\cos(t),\sin(t))
\]
\[
(\nabla g \circ \gamma)(t) = (0, 2\cos(t), 2\sin(t))
\]
Clearly, for each $t \in [0, \pi]$, $\mathrm{EL}_\gamma(t)$ and $(\nabla g \circ \gamma)(t)$ are parallel, and perpendicular to the sphere, as expected. 
\end{Example}

\begin{Example}[\textbf{An isoperimetric problem}]
Consider the set of curves $\gamma \colon [0, 2\pi] \to \mathbb{E}^2 : t \mapsto (x(t), y(t))$ which start and end at the origin, and are of length $2\pi$. The loops which maximize the area enclosed by the loop are, by Green's theorem, those which maximize
\[
\int_0^{2\pi} \frac{1}{2}(x(t)y'(t)-x'(t)y(t))\,dt
\]
which is an action with Lagrangian $L(x,y,\dot{x},\dot{y},t) = \frac{1}{2}(x\dot{y}-\dot{x}y)$.
Moreover, the length constraint amounts to
\[
\int_0^{2\pi} \sqrt{x'(t)^2 + y'(t)^2}\,dt = 1
\]
and the lefthand side here is an action with Lagrangian $M(x,y,\dot{x},\dot{y},t) = \sqrt{\dot{x}^2+\dot{y}^2}$. As is well-known, the loops which (locally) maximize the enclosed area are exactly the circular loops. For example, $\gamma(t) = (1-\cos(t), \sin(t))$ is one such loop. We find that:
\[
\mathrm{EL}_\gamma^L(t) = (\cos(t), -\sin(t))
\]
\[
\mathrm{EL}_\gamma^M(t) = (-\cos(t), \sin(t))
\]
As expected, $\mathrm{EL}_\gamma^L(t)$ and $\mathrm{EL}_\gamma^M(t)$ are parallel for each $t \in [0, 2\pi]$.
\end{Example}

\begin{Example}[\textbf{Projectiles}]
Consider a particle of unit mass moving in the plane with Lagrangian $L = \frac{1}{2}(\dot{x}^2 + \dot{y}^2) - gy$. This represents a projectile moving under gravity. For a path $\gamma(t) = (x(t), y(t))$, the action is:
\[
S_{\text{proj}}(\gamma) = \int_a^b \frac{1}{2}(x'(t)^2 + y'(t)^2) - gy(t)\,dt
\]
For $a = 0$, $b = 3$, and initial and final points $(0,0)$ and $(3,0$), respectively, we find that the minimizing path is $\gamma_0(t) = (3t, -\frac{1}{2}gt(t-1))$. As $\gamma_0$ is a minimizing path, at other paths, we expect the negative gradient $-\nabla S_{\text{proj}}$, represented by the Euler-Lagrange path, to point toward $\gamma_0$. Consider the path $\gamma_1(t) = (3t, -gt(t-1))$. We find that:
\[
\mathrm{EL}_{\gamma_1}(t) = (0, g)
\]
As in Figure \ref{fig:projectile}, we see that $-\mathrm{EL}_{\gamma_1}$ points in the expected direction, toward the minimizing path (the vectors $-\mathrm{EL}_{\gamma_1}(t)$ have been scaled down by a factor of ten, to produce a reasonable diagram).
\begin{figure}[h]
\begin{center}
\begin{tikzpicture}
\draw [-, gray] (0,-1) -- (0,4);
\draw [-, gray] (-1,0) -- (4,0);

\node[circle, minimum size=2mm, inner sep=0pt, fill=black] (A) at (0,0) {};
\node[circle, minimum size=2mm, inner sep=0pt, fill=black] (A) at (3,0) {};

\draw [black, thick,  domain=0:1, samples=40] plot ({3*\x}, {-0.5*9.8*\x*(\x-1)} );
\draw [black, thick,  domain=0:1, samples=40] plot ({3*\x}, {-9.8*\x*(\x-1)} );

\node at (1.5,2.7) (A) {$\gamma_1$};
\node at (1.5, 0.9) (C) {$\gamma_0$};

\draw [->, blue] ( 0.5 , 1.3611111111111112 ) -- ( 0.5 , 0.38111111111111107 );
\draw [->, blue] ( 1.0 , 2.177777777777778 ) -- ( 1.0 , 1.197777777777778 );
\draw [->, blue] ( 1.5 , 2.45 ) -- ( 1.5 , 1.4700000000000002 );
\draw [->, blue] ( 2.0 , 2.177777777777778 ) -- ( 2.0 , 1.197777777777778 );
\draw [->, blue] ( 2.5 , 1.361111111111111 ) -- ( 2.5 , 0.38111111111111085 );
\end{tikzpicture}
\end{center}
\caption{The blue vectors are the values of $-\mathrm{EL}_{\gamma_1}(t)$.}
\label{fig:projectile}
\end{figure}
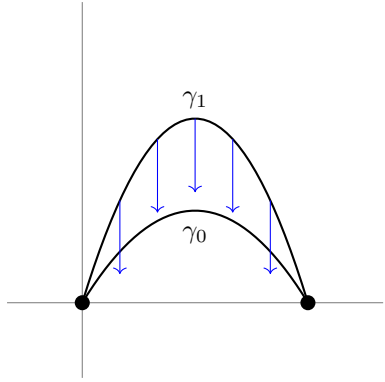
\end{Example}
 
\begin{Example}[\textbf{Harmonic oscillators}]
Consider a particle of unit mass moving in one dimension with Lagrangian $L = \frac{1}{2}\dot{x}^2 - \frac{1}{2}x^2$. This represents a harmonic oscillator. For a path $\gamma(t) = x(t)$, the action is:
\[
S_{\text{osc}}(\gamma) = \int_a^b \frac{1}{2}x'(t)^2 - \frac{1}{2}x(t)^2\,dt
\]
For $a = 0$, $b = \pi$, and initial and final points $-1$ and $1$, respectively, we find that the minimizing path is $\gamma_0(t) = -\cos(t)$. As $\gamma_0$ is a minimizing path, at other paths, we expect the negative gradient $-\nabla S_{\text{osc}}$, represented by the Euler-Lagrange path, to point toward $\gamma_0$. Consider the constant-velocity path $\gamma_1(t) = \frac{2}{\pi}t-1$. We find that:
\[
\mathrm{EL}_{\gamma_1}(t) = -\frac{2}{\pi}t+1
\]
Thus, when $t < \frac{\pi}{2}$, $-\mathrm{EL}_{\gamma_1}(t) < 0$, whereas when $t > \frac{\pi}{2}$, $-\mathrm{EL}_{\gamma_1}(t) > 0$. This, combined with Figure \ref{fig:harmonic_oscillator}, illustrates that $-\mathrm{EL}_{\gamma_1}$ points in the expected direction, toward the minimizing path.
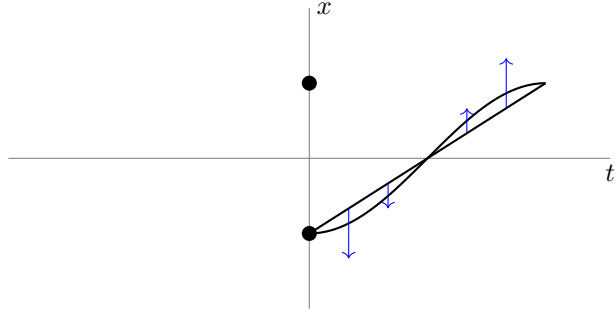
\begin{figure}[h]
\begin{center}
\begin{tikzpicture}
\draw [-, gray] (0,-2) -- (0,2);
\draw [-, gray] (-4,0) -- (4,0);

\node[circle, minimum size=2mm, inner sep=0pt, fill=black] (A) at (0,-1) {};
\node[circle, minimum size=2mm, inner sep=0pt, fill=black] (A) at (0,1) {};

\draw [black, thick,  domain=0:3.14159, samples=40] plot ({\x}, {-cos(deg(\x)} );
\draw [black, thick,  domain=0:3.14159, samples=40] plot ({\x}, {(2/3.14159)*\x-1)} );

\node at (4,-0.2) (A) {$t$};
\node at (0.2, 2) (C) {$x$};

\draw [->, blue] ( 0.5235983333333333 , -0.6666666666666667 ) -- ( 0.5235983333333333 , -1.3333333333333335 );
\draw [->, blue] ( 1.0471966666666666 , -0.33333333333333337 ) -- ( 1.0471966666666666 , -0.6666666666666667 );
\draw [->, blue] ( 2.094393333333333 , 0.33333333333333326 ) -- ( 2.094393333333333 , 0.6666666666666665 );
\draw [->, blue] ( 2.6179916666666667 , 0.6666666666666667 ) -- ( 2.6179916666666667 , 1.3333333333333335 );
\end{tikzpicture}
\end{center}
\caption{The blue vectors are the values of $-\mathrm{EL}_{\gamma_1}(t)$.}
\label{fig:harmonic_oscillator}
\end{figure}
\end{Example}

\begin{Example}[\textbf{Minimizing the loss function in a statistical learning problem}]
As a final example, we consider a problem from statistical learning. Let $X$ and $Y$ be two random variables in $[0,1]$, where we think of $X$ as an input and $Y$ as a corresponding output. The statistical relation between the input $X$ and output $Y$ can be summarized by a joint probability density function $p \colon [0,1]^2 \to \R$. This density $p$ is unknown to us. Suppose that we model the relation between $X$ and $Y$ by a function $f \colon [0,1] \to [0,1]$. To measure the accuracy of this model, we introduce a loss function $R \colon [0,1]^2 \to [0,1]$. For the purposes of this example, let us take the squared loss $R(\widehat y,y) = (\widehat y-y)^2$ (here, we think of $\widehat y$ as our prediction for $Y$ and $y$ as the true value of $Y$). Now, for our model $f$, the expected value of the loss $R(f)$ is:
\[
\mathbb{E}[R(f)] = \int_0^1\int_0^1 (f(x)-y)^2p(x,y)\,dx\,dy
\]
This is an action (of $f$) with Lagrangian:
\[
L(q,\dot{q},x) = \int_0^1 (q-y)^2p(x,y)\,dy
\]
For the corresponding Euler-Lagrange path, we find that:
\begin{align*}
\mathrm{EL}_f(x) &= \frac{\partial}{\partial q}\int_0^1 (q-y)^2p(x,y)\,dy \\
&= 2\int_0^1 (f(x)-y)p(x,y)\,dy \\
&= 2f(x)\int_0^1p(x,y)\,dy - 2\int_0^1yp(x,y)\,dy
\end{align*}
Thus, we get the intuitive result that, to minimize the expected loss, we should choose the model:
\[
f(x) = \frac{\int_0^1yp(x,y)\,dy}{\int_0^1p(x,y)\,dy} = \mathbb{E}[Y | X = x]
\]
For example, if $p \equiv 1$, which is to say if both $X$ and $Y$ are uniformly randomly distributed, then the optimal model is simply $f \equiv \frac{1}{2}$. Moreover, we see that $-\mathrm{EL}_f(x)$ is positive if $f(x)$ is smaller than the conditional expected value $\mathbb{E}[Y | X = x]$ of $Y$, and is negative if $f$ is larger than this expected value. Thus, once more, $-\mathrm{EL}_f$, being a representative for $-\nabla \mathbb{E}[R(-)]$, points in the expected direction, that of minimization.
\end{Example}

\end{document}